\documentclass[a4paper]{amsart}
\usepackage{amssymb}
\usepackage{xypic}
\usepackage{graphicx}

\newtheorem{theorem}{Theorem}[section]
\newtheorem{lemma}[theorem]{Lemma}

\newtheorem{cor}[theorem]{Corollary}
\newtheorem{conjec}[theorem]{Conjecture}

\theoremstyle{definition}

\theoremstyle{remark}
\newtheorem{remark}[theorem]{Remark}

\numberwithin{equation}{section}

\newcommand{\C}{\mathbb{C}}

\newcommand{\R}{\mathbb{R}}

\DeclareMathOperator{\Ric}{Ric}
\DeclareMathOperator{\inj}{inj}
\DeclareMathOperator{\scal}{scal}
\DeclareMathOperator{\vol}{Vol}
\DeclareMathOperator{\dvol}{dvol}
\DeclareMathOperator{\conj}{conj}
\DeclareMathOperator{\diam}{diam}

\title[Differentiable Stability and positive scalar curvature]{Differentiable stability and sphere theorems\\ for manifolds and Einstein manifolds\\ with positive scalar curvature}

\author{Wilderich Tuschmann}
\address{Institut f\"ur Algebra und Geometrie\\ Karlsruher Institut f\"ur Technologie\\ Kaiserstrasse 89--93\\ D-76133 Karlsruhe\\ Germany}
\email{wilderich.tuschmann@kit.edu}
\thanks{}

\author{Michael Wiemeler}
\address{Institut f\"ur Mathematik\\ Universit\"at Augsburg\\ D-86135 Augsburg\\ Germany}
\email{michael.wiemeler@math.uni-augsburg.de}
\thanks{}

\subjclass[2010]{}

\keywords{Positive scalar curvature, sphere theorems, stability, rigidity, Einstein manifolds}

\begin{document}
\begin{abstract} 
Leon Green obtained remarkable rigidity results
for manifolds of positive scalar curvature with large conjugate radius and/or injectivity radius.
Using $C^{k,\alpha}$ convergence techniques, we prove several differentiable stability and sphere theorem versions of these results 
and apply those also to the study
of Einstein manifolds.
\end{abstract}

\maketitle


\section{Introduction}
\label{sec:intro}

\

\noindent In 1963, Leon Green proved the following remarkable result:

\begin{theorem}[{Green \cite{MR0155271}}]
\label{sec:introduction-5}
  Let \(M\) be an \(n\)-dimensional closed Riemannian manifold whose scalar curvature and volume
	satisfy the inequality 
  \begin{equation*}
    \int_M\scal_M\dvol_M \geq n(n-1)\vol M.
  \end{equation*}
  Then the conjugate radius of \(M\) is bounded from above by \(\conj M\leq \pi\), and equality holds here 
	if and only if \(M\) is isometric to a spherical space form of constant sectional curvature one.
\end{theorem}

\

Green's theorem implies in particular the following rigidity result:

\

\begin{cor}
\label{sec:introduction-6}
A closed Riemannian $n$-manifold with scalar curvature $\scal_M \ge n(n-1)$ 
and injectivity radius equal to $\pi$ is isometric to the $n$-dimensional unit sphere $S^n(1)$.
\end{cor}

\

Thus, when replacing diameter by injectivity radius, 
this latter rigidity result can be viewed as a scalar curvature analogue to 
the well-known maximal diameter sphere theorems for sectional curvature $\ge 1$ and Ricci curvature $\ge n-1$
obtained by Toponogov \cite{MR0103510}  in 1959 and Cheng \cite{MR0378001} in 1975.

\

In view of Grove-Shiohama's celebrated diameter sphere theorem for positive sectional curvature (see \cite{MR0500705}) and the
wealth of other sphere theorems for manifolds of positive sectional and of positive Ricci curvature (see, e.g., \cite{MR0140054}, \cite{MR0139120}, \cite{MR2449060}, \cite{MR3096308},
 \cite{MR1074481}, \cite{MR1240887}, \cite{MR1369415},
\cite{MR1100210}, \cite{MR1049696}, \cite{MR697986},  \cite{MR1692012}, \cite{MR1326988}, 
 \cite{MR1209707}, \cite{MR682734}, \cite{MR1855649}, \cite{MR1001711}), it is therefore
natural to ask which conditions on the injectivity radius, or, more generally, conjugate radius, of a
closed Riemannian $n$-manifold $M$ with positive {\em scalar} curvature 
will guarantee stability of
Green's above-mentioned results in the sense that $M$ can still be recognized as being
homeomorphic, or even diffeomorphic, to the standard $n$-sphere or, respectively, to an $n$-dimensional spherical space form.

\

Of course, stability results which actually imply diffeomorphism are of much more significance in this context than merely topological ones.
This is, in particular, also due to the fact that exotic spheres with positive scalar curvature are known to abound.
Indeed, any homotopy sphere of dimension \(n\not\equiv 1,2 \mod 8\), \(n\neq 4\), admits a metric of positive scalar curvature.
In the other dimensions, for \(n\geq 9\), all homotopy spheres which bound spin manifolds admit metrics of positive scalar curvature.
These constitute half of all of the homotopy spheres in these dimensions.
These results follow from a combination of results from \cite[Corollary 2.7]{MR0219077}, \cite[Proof of Theorem 2]{MR0180978} and \cite{MR1189863}.
(For more on curvature properties of exotic spheres 
see also the survey \cite{MR2434347}.)

\

The first main result of this note and its corollaries provide positive answers to whether Green's rigidity results
are differentiably stable as follows:


\

\begin{theorem}
\label{sec:introduction}
 For all $n \in\mathbb{N}$, $C,\lambda_0,\lambda_1,i_0>0$ and $0\leq \beta < 1$ there exists 
$\epsilon=\epsilon(n,C,\lambda_0,\lambda_1,i_0,\beta)>0$ such that every closed \(n\)-dimensional  Riemannian manifold $M$ with
  \begin{align*}
    \int_M\scal_M\dvol_M&\geq n(n-1)\vol M& \conj M& \geq \pi-\epsilon\\
    \Ric_M&\geq -\lambda_0& \| \nabla\Ric_M\|&\leq \lambda_1\\
    \vol M &\leq \frac{C}{(\pi- \conj M)^\beta}& \inj M&\geq i_0
  \end{align*}
is diffeomorphic to a spherical space form.
\end{theorem}

\

\noindent To put this result into further perspective, let us mention 
that it yields in particular various new recognition and stability theorems
for closed manifolds with positive curvatures. 

First, we have the following new sphere theorem for manifolds with mean positive scalar curvature:

\

\begin{cor}
\label{sec:introduction-1}
  For all $n \in\mathbb{N}$ and $\lambda_1,d>0$ there exists \(\epsilon=\epsilon(n,\lambda_1,d)>0\) such that every closed \(n\)-dimensional  Riemannian manifold \(M\) with
  \begin{align*}
    \int_M\scal_M\dvol_M&\geq n(n-1)\vol M & \inj M&\geq \pi -\epsilon\\
    \diam M&\leq d  & \| \nabla\Ric_M\|&\leq \lambda_1
\end{align*}
is  diffeomorphic to the standard $n$-sphere.  
\end{cor}

\

Since Berger's initial studies in the 1960s (see  \cite{MR0130659}, \cite{MR0238226}), 
much work in Riemannian geometry has also been devoted to the problem
to classify all Einstein manifolds which satisfy further other curvature conditions. For example,
Berger proved in this context in particular the following isolation result for
the standard round spheres, namely: 
If $(M,g)$ is a closed simply connected Einstein manifold 
of dimension n which is strictly $3n / (7n - 4)$-pinched,
then $M$ is, up to scaling, isometric to the standard $n$-sphere
(compare \cite[Section 0.33]{MR867684}).
Moreover, Brendle showed that compact Einstein $n$-manifolds, $n\ge 4$, with positive isotropic curvature
have constant sectional curvature (see \cite{MR2573825}).

\

For Einstein manifolds with positive Einstein constant, we obtain
(for a more general version see Theorem~\ref{sec:test3} below)
the following differentiable stability result, 
which does not require any further curvature bounds:

\

\begin{cor}
\label{sec:introduction-4}
 For all $n \in\mathbb{N}$ there exists \(\epsilon=\epsilon(n)>0\) 
such that every closed simply connected \(n\)-dimensional  Einstein manifold \(M\) 
with Einstein constant \(n-1\) and conjugate radius $\conj M \geq \pi-\epsilon$
  is  diffeomorphic to the standard $n$-sphere.
\end{cor}

\

Notice here also that in \cite{MR2178969} families of inequivalent Einstein metrics with positive Einstein constant were constructed on the \((4m+1)\)-dimensional Kervaire spheres, where \(m \geq 2\).
The Kervaire spheres are known to be exotic in dimensions \(n\neq 2^k-3\), where \(k\geq 3\).
The authors of \cite{MR2178969} also conjectured that such metrics exist on all odd dimensional homotopy spheres which bound parallelizable manifolds.
This conjecture is true in dimensions less than or equal to \(15\) by \cite{MR2146519}. 

\

If we only have a positive lower bound on the Ricci-curvature we get:

\

\begin{theorem}
\label{sec:introduction-3}
 For all $n \in\mathbb{N}$ and $k >0$ there exists  \(\epsilon=\epsilon(n,k)>0\) such that every closed \(n\)-dimensional  Riemannian manifold \(M\) with
  \begin{align*}
    |\pi_1(M)|&\leq k\\
    \Ric_M&\geq n-1\\
    \conj M& \geq \pi-\epsilon
  \end{align*}
is  diffeomorphic to an Einstein manifold with Einstein constant \(n-1\).
\end{theorem}

\

We conjecture that a manifold $M$ as in the theorem above is actually also already diffeomorphic
to a spherical space form.
This conjecture is in part motivated by the following fact:
 For each \(n\in \mathbb{N}\) there exists an \(\epsilon(n)>0\) such that every closed \(n\)-manifold \(M\) with \(\inj M\geq \pi-\epsilon\) and \(\Ric_M\geq n-1\) is diffeomorphic to a sphere. 
This follows from the following argument.
 An inequality relating the injectivity radius of a manifold to its volume due to Berger \cite{MR597027} implies that the volume of a manifold with injectivity radius close to \(\pi\) is close to the volume of the round sphere with radius one.
Therefore, by a result of Cheeger and Colding \cite[Theorem A.1.10]{MR1484888}, it follows that \(M\) is diffeomorphic to a sphere.

\

For manifolds with positive sectional curvature we get:

\

\begin{cor}
\label{sec:introduction-2}
 For all $n \geq 4$ and $k,\delta >0$ there exists \(\epsilon=\epsilon(n,k,\delta)>0\) such that every closed \(n\)-dimensional  Riemannian manifold \(M\) with
  \begin{align*}
    |\pi_1(M)|&\leq k\\
    \sec_M&\geq 1-\frac{3}{n+2}+\delta\\
    \Ric_M&\geq n-1\\
    \conj M& \geq \pi-\epsilon
  \end{align*}
is  diffeomorphic to a spherical space form.
\end{cor}

\

\begin{remark}
  The lower bound on the conjugate radius of \(M\) in Theorem~\ref{sec:introduction} cannot be replaced by a lower bound on the diameter.
	
  This is due to the following fact: any manifold \(M\) of dimension at least three which admits a metric of scalar curvature \(>c>0\) admits a metric of scalar curvature \(>c'(c)>0\) and arbitrarily large diameter.
  Indeed, one can construct such a metric as follows:
  By results which were independently proven by Gromov and Lawson \cite{MR577131} and Schoen and Yau \cite{MR535700}, we can assume that \(M\) has a metric of positive scalar curvature 
	which is a 'torpedo' metric in the neighborhood of a point \(p\in M\).
  We can then choose this torpedo to be arbitrarily long without decreasing the scalar curvature.
\end{remark}

\

The main ingredients in the proofs of the above results consist in
using $C^{k,\alpha}$ convergence techniques, developed in particular
by Anderson \cite{MR1074481}  and Anderson and Cheeger \cite{MR1158336}, 
and appropriate modifications of Green's original arguments
along with convergence properties of the injectivity and conjugate radius.

\

\begin{remark}
We actually expect that one can remove the bounds on the covariant derivative of the Ricci curvature 
and the volume in Theorem~\ref{sec:introduction}
without effecting the conclusion of the theorem.
However, any sort of convergence theory for scalar curvature has still to emerge,
and all techniques which are currently at disposal require these additional bounds for the proof to work.

Namely, the bound on the covariant derivative is needed to construct a sequence of Riemannian manifolds 
which converges in \(C^{2,\alpha}\)-topology.
This strong form of convergence guarantees that certain geodesics and the conjugate radii 
of these manifolds will converge.
The volume bound is necessary for showing that a certain sequence of integrals over the tangent unit sphere bundles of the manifolds converges.
\end{remark}

\

\begin{remark}
  If one has a sequence of Riemannian manifolds which satisfy a two-sided sectional curvature bound and a lower bound on the conjugate radii, then using Ricci-flow one can smooth this sequence, so that all the manifolds in the sequence satisfy two-sided bounds on the Ricci-curvature and its covariant derivatives.
  But it is completely unclear what happens to the conjugate radii of the manifolds.
  The only thing one knows about it is that there is some lower bound.
  For the proof of Theorem \ref{sec:introduction} we need that the conjugate radii of the sequence converge to \(\pi\).
  Therefore we cannot replace the bounds on the Ricci-curvature and its covariant derivatives by a two-sided sectional curvature bound.
\end{remark}

\ 

\begin{remark}
  Recently it has been shown by Gromov \cite{MR3201312} that if a sequence of \(C^2\)-metrics \(g_i\) on a manifold with \(\scal_{g_i}\geq n(n-1)\) converges in the \(C^0\)-topology to a \(C^2\)-metric \(g\), then we have \(\scal_{g}\geq n(n-1)\).
  Since the injectivity radius is upper semi-continuous in \(C^1\)-topology, a version of Theorem~\ref{sec:introduction} which only needs a bound on the Ricci-curvature and not on its covariant derivative would follow from Green's theorem, if one can guarantee that the limit metric is at least of class \(C^2\).
\end{remark}

\

The remaining parts of the present note are structured as follows: After
recalling relevant preliminaries in Section \ref{sec:preliminaries}, Section \ref{sec:proof} is devoted
to the proofs of Theorem~\ref{sec:introduction}
and Corollary~\ref{sec:introduction-1}.
The final Section \ref{sec:app}
is concerned with the proofs of Corollary~\ref{sec:introduction-2}, Theorem~\ref{sec:introduction-3}, and a more
general version of Corollary~\ref{sec:introduction-4}.

\

It is our pleasure to thank Aaron Naber for helpful comments, especially
for pointing out to us that a reference, which we used in the first version
of this note in the proof of Theorem~\ref{sec:test5}, actually contained incorrect results
so that we had to furnish a different proof of Theorem \ref{sec:test5}. 
In addition, we would like to thank a first referee whose comments led us to
Lemma \ref{sec:preliminaries-3} and thus to a significant improvement
of the main result, Theorem \ref{sec:introduction}, 
from both-sided bounds on the Ricci curvature to now just a lower one.

\

\section{Preliminaries}
\label{sec:preliminaries}

\noindent Let us first recall Green's theorem in its original version as well as its proof,
since we will have to refer to them in our later considerations.

\begin{theorem}[{\cite[Theorem 5.1]{MR0155271}, \cite[Proposition 5.64]{MR496885}}]
\label{sec:proof-main-results}
  Let  \(M\) be a closed Riemannian manifold of dimension \(n\) with conjugate radius \(\conj M\geq a\).
  Then
  \begin{equation*}
    \vol M \geq \frac{a^2}{n(n-1)\pi^2} \int_M \scal_M\;\dvol_M.
  \end{equation*}
  Moreover, equality holds if and only if \(M\) has constant sectional curvature \(\pi^2/a^2\).
\end{theorem}
\begin{proof}
  Let \(\gamma:[0,a]\rightarrow M\) be a geodesic.
  Since \(\conj M\geq a\), the index form \(I(X,X)\) is non-negative for any vector field \(X\) along \(\gamma\) with \(X(0)=0\) and \(X(a)=0\).
  In particular, in the special case \(X(t)=\sin (\frac{\pi t}{a})V(t)\), where \(V(t)\) is a parallel vector field along \(\gamma\),
	it follows that
  \begin{equation*}
    \int_{0}^{a}\left( \frac{\pi^2}{a^2}\cos^2\frac{\pi t}{a}- K(\gamma'(t),V(t))\sin^2\frac{\pi t}{a}\right)\;dt\geq 0.
  \end{equation*}
  Let \(X_1,\dots,X_{n-1}\) be \(n-1\) pairwise orthogonal vector fields as above.
  Adding up the above integrals, we obtain:
  \begin{equation*}
    \int_{0}^{a}\left( (n-1)\frac{\pi^2}{a^2}\cos^2\frac{\pi t}{a}- \Ric(\gamma'(t),\gamma'(t))\sin^2\frac{\pi t}{a}\right)\;dt\geq 0.
  \end{equation*}
  Now we integrate this inequality over all geodesics of length \(a\). To do so, we fix some notation:
  Let \(\zeta^t\) be the geodesic flow on the unit sphere bundle \(SM\) associated to the tangent bundle of \(M\).
  Moreover, let \(\mu_1\) be the canonical measure induced by \(g\) on \(SM\).
  Note that \(\mu_1\) is invariant under \(\zeta^t\) and that \(\gamma'(t)=\zeta^t(u)\), where \(\gamma\) is the geodesic with \(\gamma'(0)=u\).
  Applying Fubini's theorem, we obtain:
 \begin{align*}
    0&\leq\int_{SM}\int_{0}^{a}\left( (n-1)\frac{\pi^2}{a^2}\cos^2\frac{\pi t}{a}- \Ric(\zeta^t(u),\zeta^t(u))\sin^2\frac{\pi t}{a}\right)\;dtd\mu_1\\
    &=\vol(M)\vol(S^{n-1})(n-1)\frac{\pi^2}{a^2}\int_{0}^a \cos^2\frac{\pi t}{a}\;dt - \int_{SM\times[0,a]} \Ric(u,u)\sin^2\frac{\pi t}{a}\;dtd\mu_1\\
    &=\vol(M)\vol(S^{n-1})(n-1)\frac{\pi^2}{a^2}\int_{0}^a \cos^2\frac{\pi t}{a}\;dt - \int_{SM} \Ric(u,u)\;d\mu_1\int_0^a \sin^2\frac{\pi t}{a}\;dt\\
    &=\vol(M)\vol(S^{n-1})(n-1)\frac{\pi^2}{a^2}\int_{0}^a \cos^2\frac{\pi t}{a}\;dt -\\
    &\quad\quad\quad \int_M\int_{S_mM} \Ric(u,u)\;d\sigma \dvol_M\int_0^a \sin^2\frac{\pi t}{a}\;dt,
  \end{align*}
  where \(\sigma\) is the canonical measure on \(S_mM\) and \(\vol(S^{n-1})\) is the volume of the \((n-1)\)-dimensional sphere of radius one.
	
  A general formula for quadratic forms on Euclidean space yields
  \begin{equation*}
    \int_{S_mM} \Ric(u,u)\;d\sigma=\vol(S^{n-1})\frac{1}{n}\scal_M(m).
  \end{equation*}
  Now the first claim follows since
  \begin{equation*}
    \int_0^a\cos^2{\frac{\pi t}{a}}\;dt=\int_0^a\sin^2{\frac{\pi t}{a}}\;dt.
  \end{equation*}

  In the case where one actually has equality in the above inequality, it follows that \(I(X_i,X_i)=0\) for all \(i\).
  Therefore all vector fields  \(X_i\) are Jacobi fields, and hence
  it follows that \(M\) has constant curvature \(\pi^2/a^2\).
\end{proof}

As corollaries to the above theorem, one immediately obtains Theorem~\ref{sec:introduction-5} and Corollary~\ref{sec:introduction-6}.


\

We also recall the concept of \(C^{k,\alpha}\) convergence for pointed complete Riemannian manifolds.
Further details on these concepts can be found in, e.g., \cite[Chapter 10]{MR2243772}.

A sequence of pointed complete Riemannian manifolds \((M_i,p_i,g_i)\) converges in the pointed \(C^{k,\alpha}\) topology to
 a pointed manifold \((M_\infty,p_\infty,g_\infty)\) 
if for every \(R>0\) there is a domain \(\Omega\supset B_R(p_\infty)\subset M_\infty\) and if, 
for large \(i\), there are embeddings \(F_i:\Omega\rightarrow M_i\) such that \(F_i(\Omega)\supset B_R(p_i)\) and 
such that the pull-backs \(F_i^*g_i\) converge on \(\Omega\) to \(g_\infty\) in the \(C^{k,\alpha}\) H\"older topology.
If there is an upper bound \(d\) for the diameters of all of the \(M_i\), then, for \(R>d\), \(B_R(p_\infty)\) 
(and therefore also \(M_\infty\)) will be diffeomorphic to \(M_i\) if \(i\) is sufficiently large.
Therefore, in this case one can also speak about mere (unpointed) \(C^{k,\alpha}\) convergence.

\

Anderson proved a \(C^{k,\alpha}\) precompactness theorem for manifolds with bounded Ricci curvature.
To state it, we introduce the following notation:

For \(n\in \mathbb{N}\), \(i_0>0\), \(\lambda_0,\dots,\lambda_k>0\), we denote by \(\mathcal{M}(n,i_0,\lambda_0,\dots,\lambda_k)\) the class of \(n\)-dimensional pointed complete Riemannian manifolds \(M\) with
\begin{align*}
  \|\nabla^i\Ric_M\| \leq \lambda_i & \text{ for } i=0,\dots,k\\
  \inj M \geq i_0.
\end{align*}

We can now state Anderson's result.

\begin{theorem}[\cite{MR1074481}]
\label{sec:preliminaries-1}
  The class \(\mathcal{M}(n,i_0,\lambda_0,\dots,\lambda_k)\) is precompact in the pointed \(C^{k+1,\alpha}\) topology.
\end{theorem}

In the case where there is only a lower Ricci curvature bound, the
following precompactness result has been shown by Anderson and Cheeger \cite{MR1158336}.

\begin{theorem}
\label{sec:preliminaries-2}
  Let \(\lambda \in \R\) and \(i_0,d>0\). Denote by \(\mathcal{N}(n,i_0,d,\lambda)\) the class of \(n\)-dimensional closed Riemannian manifolds \(M\) with \(\Ric_M\geq\lambda\), \(\inj M \geq i_0\) and \(\diam M \leq d\).

Then the class \(\mathcal{N}(n,i_0,d,\lambda)\) is precompact in the \(C^{0,\alpha}\) topology.
\end{theorem}

For the application of the above compactness results in our situation we need the following lemma.

\begin{lemma}
\label{sec:preliminaries-3}
  Let \(\lambda_1,i_0>0\) and \(n\in \mathbb{N}\). Then there is a constant \(\lambda_0(\lambda_1,i_0,n)\) such that for every \(n\)-dimensional complete  Riemannian manifold \(M\) with \(\|\nabla \Ric_M\|\leq \lambda_1\) and \(\conj M\geq i_0\) one has \(\Ric_M\leq \lambda_0\).
\end{lemma}
\begin{proof}
  Assume that there are \(x\in M\) and \(Y\in T_xM\) with \(\|Y\|=1\) and \(\Ric(Y,Y)> (n-1)\frac{\pi^2}{i_0^2} + \lambda_1 i_0=\lambda_0\).
  Let \(\gamma:[0,i_0]\rightarrow M\) be the geodesic with \(\gamma(0)=x\) and \(\gamma'(0)=Y\). Then it follows from the bound on the covariant derivative of the Ricci-tensor and the mean value theorem, that, for all \(t\in [0,i_0]\), \(\Ric(\gamma'(t),\gamma'(t))> (n-1)\frac{\pi^2}{i_0^2}\).
  Therefore it follows as in the proof of Myer's diameter  estimate \cite[p. 171]{MR2243772} that the first conjugate point of \(x\) along \(\gamma\) has distance less than \(i_0\) to \(x\).
  This contradicts the assumption that \(\conj M \geq i_0\).
\end{proof}

\section{Proof of the main result}
\label{sec:proof}

Now we can prove our main theorem.

\ 

\begin{theorem} 
\label{sec:test}
 For all $n \in\mathbb{N}$, $C,\lambda_0,\lambda_1,i_0>0$ and $0\leq \beta < 1$ there exists 
$\epsilon=\epsilon(n,C,\lambda_0,\lambda_1,i_0,\beta)>0$ such that every closed \(n\)-dimensional  Riemannian manifold $M$ with
  \begin{align*}
    \int_M\scal_M&\geq n(n-1)\vol M& \conj M& \geq \pi-\epsilon\\
    \Ric_M&\geq -\lambda_0& \| \nabla\Ric_M\|&\leq \lambda_1\\
    \vol M &\leq \frac{C}{(\pi- \conj M)^\beta}& \inj M&\geq i_0
  \end{align*}
is diffeomorphic to a spherical space form.
\end{theorem}

\

\begin{proof}
  Assume that the theorem is not true.
  Then, for some $n \in\mathbb{N}$, $C,\lambda_0,\lambda_1,i_0>0$ and $0\leq \beta < 1$,
	there is a sequence of closed \(n\)-dimensional Riemannian manifolds \((M_i,g_i)\), such that
 \begin{align*}
   \int_{M_i}\scal_{M_i}\dvol_{M_i}&\geq n(n-1)\vol M_i& \lim_{i\to \infty}\conj M_i& = \pi\\
    \Ric_{M_i}&\geq -\lambda_0& \| \nabla\Ric_{M_i}\|&\leq \lambda_1\\
    \vol M_i &\leq \frac{C}{(\pi-\conj M_i)^\beta}& \inj M_i&\geq i_0   
  \end{align*}
and such that none of the \(M_i\) are diffeomorphic to spherical space forms. 
  Note that by Lemma~\ref{sec:preliminaries-3}, we may assume that \(\|\Ric_{M_i}\|\leq \lambda_0\).

By Theorem \ref{sec:preliminaries-1}, after choosing points $p_i\in M_i$ 
and passing to a subsequence if necessary, one may assume that the 
manifolds \((M_i,p_i)\) converge in the pointed \(C^{2,\alpha}\) topology to a pointed smooth manifold 
\((M_\infty,p,g_\infty)\) equipped with a \(C^{2,\alpha}\) metric $g_\infty$.

Let \(R>0\) and \(B_R(p)\subset M_{\infty}\) the ball of radius \(R\) in \(M_\infty\).
Then we may assume that the metrics \(g_i\) are defined on \(B_{R+\pi}(p)\) and that they converge in \(C^{2,\alpha}\) topology to \(g_\infty\).

Let \(x\in B_R(p)\) and \(v\in T_xM_{\infty}-\{0\}\).
Then the unit speed geodesics \(\gamma_i\) with respect to the metrics \(g_i\) with initial values \(\gamma_i(0)=x\) and \(\gamma_i'(0)=v/\sqrt{g_i(v,v)}\) and length less than or equal to \(\pi\) converge uniformly in \(C^1\) topology to a geodesic \(\gamma_\infty\) for the metric \(g_{\infty}\).

Moreover, the Ricci curvatures of \(g_i\) converge uniformly to the Ricci curvatures of the limit metric \(g_{\infty}\).
Hence, \(\Ric_{M_i}(\gamma_i'(t),\gamma_i'(t))\), \(t\in[0,\pi]\), converges uniformly to \(\Ric_{M_\infty}(\gamma_\infty'(t),\gamma_\infty'(t))\).

Let \(a_i=\conj M_i\). 
Then \(\pi=\lim_{i\to\infty} a_i= a_{\infty}\) \cite{MR0417977}, and
by the above discussion we have:
\begin{equation*}
  \begin{split}
  \lim_{i\to \infty}\int_{0}^{a_i}\left(\frac{\pi^2}{a_i^2}(n-1)-\Ric_{M_i}(\gamma_i'(t),\gamma_i'(t))\right)\sin^2 \frac{\pi}{a_i}t\;\;dt=\\\int_{0}^\pi\left((n-1)-\Ric_{M_\infty}(\gamma_\infty'(t),\gamma_\infty'(t))\right)\sin^2 t\;\;dt \geq 0    
  \end{split}
\end{equation*}

Moreover, as in the proof of Green's theorem one now computes:
\begin{equation*}
  \begin{split}
  0&\leq\int_{SM_i}\int_{0}^{a_i}\left(\frac{\pi^2}{a_i^2}(n-1)-\Ric_{M_i}(\zeta_i^t(u),\zeta^t_i(u))\right)\sin^2 \frac{\pi}{a_i}t\;\;dt\;d\mu_i\\
  &=\left(\vol(M_i)\frac{\pi^2}{a_i^2}(n-1)-\frac{1}{n}\int_{M_i}\scal_{M_i}\dvol_{M_i}\right)\vol(S^{n-1})\int_0^{a_i}\sin^2\frac{\pi}{a_i}t\;\;dt\\
  &\leq \vol (M_i)\frac{\pi^2-a_i^2}{a_i^2}(n-1)\vol(S^{n-1})\int_0^{a_i}\sin^2\frac{\pi}{a_i}t\;\;dt\\
  &\leq \frac{C}{(\pi-a_i)^\beta}\frac{\pi^2-a_i^2}{a_i^2}(n-1)\vol(S^{n-1})\int_0^{a_i}\sin^2\frac{\pi}{a_i}t\;\;dt\rightarrow 0
  \end{split}  
\end{equation*}
as \(i\to \infty\).
Here \(\mu_i\) denotes the canonical measure induced by \(g_i\) on the unit sphere bundle over \(M_i\) associated to \(TM_i\),
 and \(\vol(S^{n-1})\) stands for the volume of the \((n-1)\)-dimensional unit sphere.

Since \(\mu_i\) converges on \(SM_i|_{B_R(p)}\) to \(\mu_\infty\), it follows that
\begin{equation*}
 A= \int_{0}^\pi\left((n-1)-\Ric_{M_\infty}(\gamma_\infty'(t),\gamma_\infty'(t))\right)\sin^2 t\;\;dt=0.
\end{equation*}

Let \(V\) be a parallel vector field along \(\gamma_{\infty}\) with \(V\perp \gamma'\). Let \(W(t)=V(t)\sin t\).
Then the index form of \(W\) is bounded above by \(A\). Therefore \(W\) is a Jacobi field.
Hence it follows that \(g_\infty\) has constant sectional curvature \(1\) on \(B_R(p)\).
Since \(R\) can be chosen to be arbitrarily large, it follows that \((M_\infty,g_\infty)\) has constant curvature \(1\).
Therefore \(M_\infty\) is diffeomorphic to a spherical space form.

This is a contradiction and, hence, the proof is complete.
\end{proof}

\

We now give a proof of 
Corollary~\ref{sec:introduction-1}:

\

\begin{cor}
\label{sec:test2}
  For all $n \in\mathbb{N}$ and $\lambda_1,d>0$ there exists \(\epsilon=\epsilon(n,\lambda_1,d)>0\) such that every closed \(n\)-dimensional  Riemannian manifold \(M\) with
  \begin{align*}
    \int_M\scal_M\dvol_M&\geq n(n-1)\vol M & \inj M&\geq \pi -\epsilon\\
    \diam M &\leq d & \| \nabla\Ric_M\|&\leq \lambda_1
  \end{align*}
is  diffeomorphic to the standard $n$-sphere.  
\end{cor}

\

\begin{proof}
  At first we show that we have a two-sided Ricci curvature bound for the considered manifolds.
  For this we may assume that \(\inj M\geq \frac{\pi}{2}\).
  By Lemma~\ref{sec:preliminaries-3}, we have an upper bound for the Ricci curvature on \(M\).
  The bound on the covariant derivative of the Ricci curvature, implies that the norm of the derivative of the scalar curvature on \(M\) is bounded by a constant which only depends on \(\lambda_1\).
  Moreover, by the lower bound on the integral of the scalar curvature, we get that there is a point \(x\in M\) with \(\scal(x)\geq n(n-1)\).
  Therefore a global lower bound on the scalar curvature follows from the mean value theorem and the upper diameter bound.
  This bound only depends on \(\lambda_1\) and \(d\).
  From the upper bound on the Ricci curvature, it now follows that there is also a lower bound on the Ricci curvature.

  This lower bound on Ricci curvature and the Bishop--Gromov Volume Comparison Theorem imply 
	that the upper diameter bound of $M$ also yields an upper volume bound \(v\) for \(M\).
  Therefore we can take \(C=v\) and \(\beta=0\) and apply the theorem.
  This shows that \(M\) is diffeomorphic to an $n$-dimensional spherical space form with injectivity radius \(\pi\).
  However, the only such space form is the standard $n$-sphere.
\end{proof}

\

We also have the following corollary.

\

\begin{cor}
  For all $n \in\mathbb{N}$ and $\lambda_1,v>0$ there exists \(\epsilon=\epsilon(n,\lambda_1,v)>0\) such that every closed \(n\)-dimensional  Riemannian manifold \(M\) with
  \begin{align*}
    \int_M\scal_M\dvol_M&\geq n(n-1)\vol M & \inj M&\geq \pi -\epsilon\\
    \vol M &\leq v & \| \nabla\Ric_M\|&\leq \lambda_1
  \end{align*}
is  diffeomorphic to the standard $n$-sphere.  
\end{cor}

\

\begin{proof}
  In view of Corollary~\ref{sec:test2}, it is sufficient to show that the diameter of all the manifolds, which satisfy the bounds in the statement of the corollary, is bounded by a constant \(d(v,n)\).
  We may assume that \(\inj M\geq \pi/2\).
  Then it follows from \cite[Proposition 14]{MR608287} that the volume of geodesic balls of radius \(\pi/4\) is bounded from below by a constant \(C(n)>0\)  which only depends on the dimension of the manifold.
Let \(\gamma:[0,d]\rightarrow M\) be a minimizing geodesic between two points in \(M\) which realize the diameter of \(M\) parametrized by arc-length.
Then \(\gamma([0,d])\) can be covered by \([2d/\pi]\) disjoint geodesic balls of radius \(\pi/4\).
Therefore we have
\begin{equation*}
  \vol M \geq [2d/\pi] C(n).
\end{equation*}
Hence it follows that the diameter of \(M\) is bounded from above by some \(d(v,n)\).
\end{proof}

\section{Applications}
\label{sec:app}

In this section we apply the results from the previous section to the study of manifolds with positive Ricci 
and positive sectional curvature as well as to Einstein manifolds.

\

Our first result below directly implies the sphere theorem Corollary~\ref{sec:introduction-4} in the introduction:

\

\begin{theorem}
\label{sec:test3}
For all $n, k \in\mathbb{N}$ there exists \(\epsilon=\epsilon(n,k)>0\) 
such that every closed \(n\)-dimensional  Einstein manifold \(M\) 
with Einstein constant \(n-1\), order of the fundamental group $ |\pi_1(M)| \leq k $ 
and conjugate radius $\conj M \geq \pi-\epsilon$
is  diffeomorphic to an $n$-dimensional spherical space form.
\end{theorem}

\

\begin{proof}
  Notice first that \(\int_M\scal_M\dvol_M \geq n(n-1)\vol M\) and that \(\vol M \leq \vol S^n\).
  Moreover, we may assume that \(\conj M \geq \frac{\pi}{2}\).

  Let \(\tilde{M}\) be the Riemannian universal covering of \(M\).
  By \cite[Theorem 3.1]{MR1692012}, there is an \(i_0>0\) which does only depend on the dimension
	and the lower bounds for the Ricci curvatures and the conjugate radius of  \(\tilde{M}\),
	such that \(\inj \tilde{M}\geq i_0\).
  Since \(\tilde{M}\rightarrow M\) is a finite $k$-sheeted covering, 
	we have \(\inj M\geq \frac{1}{k}\inj \tilde{M}\geq \frac{i_0}{k}\).

  Since \(\|\nabla \Ric_M\|=0\), we can now apply Theorem~\ref{sec:introduction}, which then gives the desired conclusion.
\end{proof}

\

If we do not assume that the Ricci curvature is constant, we obtain Theorem~\ref{sec:introduction-3}
from the introduction.

\

\begin{theorem}
\label{sec:test5}
 For all $n \in\mathbb{N}$ and $k >0$ there exists  \(\epsilon=\epsilon(n,k)>0\) such that every closed \(n\)-dimensional  Riemannian manifold \(M\) with
  \begin{align*}
    |\pi_1(M)|&\leq k\\
    \Ric_M&\geq n-1\\
    \conj M& \geq \pi-\epsilon
  \end{align*}
is  diffeomorphic to an Einstein manifold with Einstein constant \(n-1\).
\end{theorem}

\

\begin{proof}
  Let \((M_i,g_i)\) be a sequence of closed Riemannian $n$-manifolds with \(\Ric_{M_i}\geq n-1\), \(a_i=\conj M_i\) and \(\lim_{i\to \infty} a_i=\pi\).
  
  As in the proof of the previous theorem, one sees that there is an \(i_0>0\) such that \(\inj M_i>i_0\).
  Therefore, by Theorem~\ref{sec:preliminaries-2}, we may assume that the \(M_i\) 
	converge in \(C^{0,\alpha}\) topology to a manifold \((M_\infty,g_\infty)\).
  In particular, we may assume that \(M_\infty\) is diffeomorphic to \(M_i\) for sufficiently large \(i\) and that the metrics \(g_i\) converge in \(C^{0,\alpha}\) topology to the metric \(g_\infty\).

  By Green's theorem, we now have
  \begin{equation*}
    n(n-1)\leq \frac{1}{\vol M_i}\int_{M_i}\scal_{M_i}\;\dvol_{M_i}\leq \frac{n(n-1)\pi^2}{a_i^2} \rightarrow n(n-1) \text{ as } i \to \infty.
  \end{equation*}
Hence for a vector field \(X\)  on \(M_\infty\) with \(\|X\|_\infty \leq 1\) and \(i\) sufficient large, we have
\begin{align*}
 & \frac{1}{\vol M_i}\int_{M_i} |\Ric_{M_i}(X,X)-(n-1)\|X\|_i^2| \;\dvol_{M_i}\\
&= \frac{1}{\vol M_i}\int_{M_i} \Ric_{M_i}(X,X)-(n-1)\|X\|_i^2 \;\dvol_{M_i}\\
&\leq \frac{1}{\vol M_i}\int_{M_i} (\scal_{M_i}-n(n-1))\|X\|_i^2 \;\dvol_{M_i}\\
&\leq \frac{2}{\vol M_i}\int_{M_i} \scal_{M_i}-n(n-1) \;\dvol_{M_i} \rightarrow 0 \text{ as } i \to \infty.
\end{align*}

Here the first inequality follows because, for \(x\in M\) and all \(Y\in T_xM_i\), 
$$\Ric_{M_i}(Y,Y)-(n-1)\|Y\|_i^2\geq 0.$$
 Moreover we have \(\scal_{M_i}(x)=\sum_{j=1}^n \Ric_{M_i}(Y_j,Y_j)\), where \(Y_1,\dots,Y_n\) is a orthonormal basis of \(T_xM_i\) with \(X(x)=\|X(x)\|_iY_1\).
Therefore it follows that
\begin{align*}
  \scal_{M_i}(x)\|X(x)\|_i-n(n-1)\|X(x)\|_i&=\sum_{j=1}^n (\Ric_{M_i}(Y_j,Y_j)\|X(x)\|_i^2-(n-1)\|X(x)\|_i^2)\\ & \geq \Ric_{M_i}(X(x),X(x))-(n-1)\|X(x)\|_i^2.
\end{align*}

Hence, we have \(\lim_{i\to\infty} \|\Ric_{M_i} - (n-1)g_i\|_{L^1}=0\).
Now, as in \cite[Proof of Proposition 3.29, Step 3]{MR1240887},
it follows from elliptic regularity theory that \(g_\infty\) is a smooth Einstein metric with Einstein constant \(n-1\). 

For the convenience of the reader, we give an outline of this argument.
Let \(h\in L^{1,p}\) for \(p\) large with \(\|h\|_{L^{1,p}}=1\).
By the H\"older inequality and the Sobolev embedding theorem, we have
\begin{equation*}
  \begin{split}
  \int_{M_\infty}\langle
  \Ric_{M_i}-(n-1)g_i,h\rangle\;\dvol_{M_\infty}\leq
  \|\Ric_{M_i}-(n-1)g_i\|_{L^1}\|h\|_{L^{\infty}}\\ \leq
  K(p,n)\|\Ric_{M_i}-(n-1)g_i\|_{L^1}\|h\|_{L^{1,p}}\rightarrow 0    
  \end{split}
\end{equation*}
as \(i\to \infty\).
Therefore it follows that \(g_\infty\) is a weak \(L^{1,p}\) solution
of the Einstein equation.

In harmonic coordinates, the Einstein equation is given by the elliptic
system

\begin{equation*}
  \label{eq:2}
  g_\infty^{kj}\frac{\partial^2 g_{\infty,rs}}{\partial x_k\partial
    x_j}+Q(g_\infty,\partial g_\infty)=(\Ric_{M_\infty})_{rs}=(n-1)g_{\infty,rs},
\end{equation*}

where \(Q\) is a quadratic term in \(g_\infty\) and the first
derivatives of \(g_\infty\).

Since \(g_\infty\in C^{0,\alpha}\cap L^{1,p}\) for large \(p\), we have locally
uniform \(C^{0,\alpha}\) bounds  on \(g_\infty^{kj}\) and \(L^{p/2}\)
bounds on \(Q\) and \(g_{\infty,rs}\).

Elliptic regularity gives uniform bounds on
\(\|g_{\infty}\|_{L^{2,p/2}}\), for large \(p\).
Therefore \(g_\infty\) is contained in \(L^{2,p/2}\cap C^{1,\alpha}\).
By continuing in this process, we get from elliptic regularity that
\(g_\infty\) is real analytic in harmonic coordinates.

Hence, \(g_\infty\) is a smooth Einstein metric.
\end{proof}

\

From the proof of the above result one also obtains what is Corollary~\ref{sec:introduction-2}
in the introduction:

\

\begin{cor}
\label{sec:test4}
 For all $n \geq 4$ and $k >0$ there exists \(\epsilon=\epsilon(n,k)>0\) such that every closed \(n\)-dimensional  Riemannian manifold \(M\) with
  \begin{align*}
    |\pi_1(M)|&\leq k\\
    \sec_M&\geq 1-\frac{3}{n+2}\\
    \Ric_M&\geq n-1\\
    \conj M& \geq \pi-\epsilon
  \end{align*}
is  diffeomorphic to a spherical space form or a locally symmetric manifold finitely covered by \(\C P^m\), \(n=2m\).

If \(n=4\), we can replace the lower bound on the sectional curvature by the bound \(\sec_M\geq \frac{2-\sqrt{2}}{2}\) to get the same conclusion.
\end{cor}

\

\begin{proof}
  Consider, as in the proof of Theorem~\ref{sec:introduction-3}, a sequence \((M_i,g_i)\) of $n$-manifolds satisfying the above bounds.
  Then a subsequence will converge in \(C^{0,\alpha}\) topology to an Einstein manifold \((M_\infty,g_\infty)\)  
	with Einstein constant \(n-1\).

  Because \(\sec_{M_i}\geq 1-\frac{3}{n+2}\), we have \(\sec_{M_\infty}\geq 1-\frac{3}{n+2}\).
  Therefore it follows from Theorem 1.1 (ii) of \cite{MR3157995} that \(M_\infty\) is isometric to a locally symmetric manifold finitely covered by \(S^n\) or \(\C P^m\).

  If \(M\) has dimension four, we get the same conclusion with the improved lower bound on the sectional curvature from results of de Araujo Costa \cite[Theorem 1.2 (b)]{MR2078673}.

  Hence the corollary follows.
\end{proof}

\

\begin{remark}
  If we assume, in the situation of Corollary \ref{sec:test4}, that \(sec_M\geq 1-\frac{3}{n+2}+\delta\), with \(\delta>0\), then by Corollary 3.1 of \cite{MR3157995} \(M\) will be diffeomorphic to a spherical space form.
  Of course, in this case \(\epsilon\) will also depend on \(\delta\).

  In dimension three a statement similar to Corollary \ref{sec:test4} follows from work of Hamilton \cite{MR664497}.

  It is asserted in \cite{cao14:_einst} that a four-dimensional Einstein-manifold with Einstein constant \(3\) and sectional curvature bounded from below by \(\frac{1}{4}\) is isometric to a locally symmetric manifold.
Given this, one can further improve the lower bound on the sectional curvature in the four-dimensional case of our corollary.
\end{remark}

\

Theorems \ref{sec:test3} and~\ref{sec:test5} also suggest the following conjecture:

\

\begin{conjec}
  Under the assumptions of Theorem~\ref{sec:test5}, \(M\) is actually  diffeo\-morphic to a spherical space form.
\end{conjec}

The problem which occurs if one wants to prove this conjecture is the following: 
In the proof of Theorem~\ref{sec:test5} one only has convergence in \(C^{0,\alpha}\) topology.
The conjugate radius is not lower semi-continuous in this topology.
 If it were upper semi-continuous in this topology, then the conjecture would follow from Theorem~\ref{sec:test5} and Green's Theorem.

\bibliography{sphere}{}
\bibliographystyle{amsplain}
\end{document}